\documentclass[pdflatex,sn-mathphys-num]{sn-jnl}

\usepackage{graphicx}%
\usepackage{multirow}%
\usepackage{amsmath,amssymb,amsfonts}%
\usepackage{amsthm}%
\usepackage{mathrsfs}%
\usepackage[title]{appendix}%
\usepackage{xcolor}%
\usepackage{textcomp}%
\usepackage{manyfoot}%
\usepackage{booktabs}%
\usepackage{algorithm}%
\usepackage{algorithmicx}%
\usepackage{algpseudocode}%
\usepackage{listings}%

\theoremstyle{thmstyleone}%
\newtheorem{thm}{Theorem}%
\newtheorem{proposition}[thm]{Proposition}%
\newtheorem{lem}[thm]{Lemma}%
\newtheorem{cor}[thm]{Corollary}%

\theoremstyle{thmstyletwo}%
\newtheorem{exa}{Example}%
\newtheorem{rk}{Remark}%
\newtheorem{property}{Property}%

\theoremstyle{thmstylethree}%
\newtheorem{dfn}{Definition}%

\def\Cl{\operatorname{Cl}}
\def\Int{\operatorname{Int}}
\def\conv{{\operatorname{conv}}}
\def\comp{{\operatorname{comp}}}

\DeclareGraphicsExtensions{.pdf}

\begin{document}

\title[Curves in hyperspaces obtained by intersection of $r$-neighborhoods with a fixed subset]{Curves in hyperspaces obtained by intersection of $r$-neighborhoods with a fixed subset}

\author*[1]{\fnm{Arsen} \sur{Galstyan}}\email{ares.1995@mail.ru}
\equalcont{These authors contributed equally to this work.}

\author[2]{\fnm{Alexey} \sur{Tuzhilin}}\email{tuz@mech.math.msu.su}
\equalcont{These authors contributed equally to this work.}

\affil*[1]{\orgdiv{School of Mathematics}, \orgname{Harbin Institute of Technology}, \orgaddress{\street{92 West Dazhi St.}, \city{Harbin}, \postcode{150001}, \state{Heilongjiang}, \country{China}}}

\affil[2]{\orgdiv{Mechanics and Mathematics}, \orgname{Lomonosov Moscow State University}, \orgaddress{\street{1 Leninskiye Gory St.}, \city{Moscow}, \postcode{119991}, \country{Russian Federation}}}

\abstract{The present paper generalizes the result from one of the papers by Galstyan. Namely, we consider two nonempty subsets $A$ and $B$ of a metric space $X$, and construct one-parametric family $F_r$ of subsets obtained by intersection between $B$ and closed $r$-neighborhood of $A$, where $r$ is bigger than the infimum distance between the sets $A$ and $B$. In the case where $B$ is compact, we show that this intersection, considered as a mapping, is right semicontinuously on $r$ in the topology generated by Hausdorff distance. Moreover, if $A$ and $B$ are convex subsets of a normed space $X$, then we prove that $F_r$ depends continuously on $r$ in such topology if and only if the Hausdorff distance between different sets $F_r$ is finite. We also show that for normed spaces $X$ of dimension $2$ or less, the latter condition is automatically fulfilled. For dimension $3$ and hence for bigger ones, we construct an example in which the Hausdorff distance between different $F_r$ is always infinite.}

\keywords{hyperspace, Hausdorff distance, set-valued mapping, continuity, convex set, normed space}

\pacs[MSC Classification]{46B20, 51F99, 52A07, 52A40, 52A41, 46B50}

\maketitle

\section{Introduction}\label{sec1}
Let $A$ and $B$ be nonempty subsets of a metric space $X$. We put $F_r(A, B)=B_r(A)\cap B$, where $B_r(A)$ is the closed $r$-neighborhood of the set $A$ (see Definition~\ref{dfn:balls}). In this work, we study in which cases the mapping $F_r$ defines a curve in the hyperspace over $X$.

Here, the hyperspace over $X$ is understood as the set $\mathcal{P}_0(X)$ of all nonempty subsets of $X$, endowed with the Hausdorff distance $d_H$ (see Definition~\ref{distance_H}), and a curve in the hyperspace is a continuous mapping from an interval to this hyperspace.

In metric geometry, many standard functionals like the distance between points or from a point to a set, are $1$-Lipschitz, that guarantees their continuity. However, in our case, the mapping $F_r$ might not possess this property. Indeed, consider the case of the plane with the Euclidean norm. When $r$ is changed by a small positive value $\delta>0$, the Hausdorff distance between $F_r$ and $F_{r+\delta}$ may change by a significantly larger value than $\delta$, see Figure~\ref{Examp :image}.
\begin{figure}[h]
\centering
\includegraphics[width=0.8\textwidth]{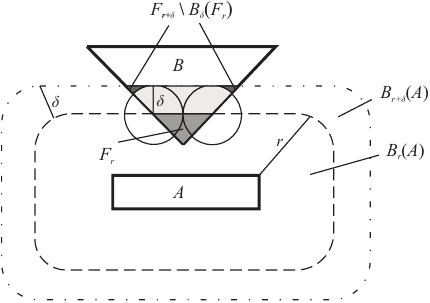}
\caption{An example where $d_H\bigl(F_r,F_{r+\delta}\bigr)>\delta$}\label{Examp :image}
\end{figure}

This figure shows two subsets $A$ and $B$. The first is a closed rectangle, the second is a closed triangle. In this example, $F_{r+\delta}\not\subset B_{\delta}\bigl(F_r\bigr)$, and thus $d_H\bigl(F_r,F_{r+\delta}\bigr)>\delta$. Moreover, the mapping $F_r$ can be discontinuous, and there are many different reasons of that, see Examples~\ref{ex:first_normed}, \ref{ex:second}, and~\ref{ex:infinity_distance}. However, it turns out that the mapping $F_r$ is right semicontinuous in an arbitrary metric space provided the subset $B$ is compact, see Theorem~\ref{thm:rightSemicont}. Notice that the compactness condition of $B$ is essential in Theorem~\ref{thm:rightSemicont} as it is illustrated in Examples~\ref{ex:first} and~\ref{ex:first_normed}. However, the compactness condition of $B$ does not imply left semicontinuity even in Euclidean plane, see Example~\ref{ex:second}. This example is based on the absence of convexity of the set $B$.

Assume now that $A$ and $B$ are convex subsets of a normed space $X$. Is it always true that $F_r$ depends continuously on $r$? We show that the only obstacle for that is the possibility of infinite distance between different $F_r$, see Theorem~\ref{cont_criterion}. Moreover, the infinity of this distance can really occur as we show in Example~\ref{ex:infinity_distance}. However, it is amazing that for $X$ of dimension $2$ or less, the distance between different $F_r$ is always finite, see Theorem~\ref{thm_lem_2}. Notice that Theorem~\ref{cont_criterion} follows from more strict Theorem~\ref{Lipschitz}: the finiteness condition implies that the $F_r$ is Lipschitz on each closed interval where the finiteness occurs.

Initially, the mapping $F_r$ arose as a byproduct of the theory of extremal networks in hyperspaces~\cite{Gals_1}. Many results in this theory use various deformations of sets of the form $B_r(A)\cap B$. However, the mapping $F$ turned out to be interesting in itself, so the proof of continuity of this mapping in the case of compact convex subsets of a finite-dimensional normed space was published separately~\cite{Gals_0}. The present paper is devoted to generalization of Theorem~3.2 from~\cite{Gals_0} to the cases of arbitrary metric spaces, arbitrary normed spaces, and for sets $A$ and $B$ of more general forms.

Note that the intersection of an $r$-neighborhood with a subset appear, for example, in~\cite[pp.~53--54]{Alimov}, whose author discusses under what circumstances the set $B_r(x)\cap B$ is a retract of the ball $B_r(x)$, where $x$ is a point in a real normed space.

On one hand, as we mentioned above, the present work relates to the theory of extremal networks, which constitutes a large part of metric geometry~\cite{Burago}. This theory investigates various variational problems, which have interested specialists for centuries. A detailed historical overview and a summary of modern results can be found in~\cite{Branching, Cieslik, Review, Hwang}.

On the other hand, this article investigates the properties of mappings in hyperspaces. The theory of hyperspaces dates back to the works of F.~Hausdorff and L.~Vietoris at the beginning of the 20th century. This theory gained popularity due to the paper~\cite{Kelley} by J.~Kelley. A historical excursion on the theory of hyperspaces can be found, for example, in~\cite{Nadler}. Mappings to hyperspaces are also called set-valued mappings. The properties of such mappings are studied by a branch of mathematics called set-valued analysis~\cite{SetAn}.

The present work also investigates the geometric structure of convex sets. Such sets and their various transformations arise in functional analysis, probability theory, information theory, and linear programming. Indeed, any system of linear inequalities in $\mathbb{R}^n$ generates some convex (possibly non-compact) polyhedron. Such systems are widely used, for example, in optimization problems, from planning and organizing production~\cite{Kant} to scheduling the work of computer processors~\cite{Drag}.

Section~\ref{Definitions} provides all the necessary definitions and auxiliary statements used in this work. Section~\ref{sec:cont_metric} investigates the continuity of the mapping $F_r$ in a metric space $X$. In Section~\ref{sec:cont_normed}, continuity criterion of $F_r$ in a real normed space is derived, see Theorem~\ref{cont_criterion}. Example~\ref{ex:infinity_distance} from Section~\ref{sec:finite_dist} shows that in Theorem~\ref{Lipschitz} and its applications such as Corollary~\ref{cont} and Theorem~\ref{cont_criterion}, the requirement of finiteness of distances between the images of $F_r$ is, in general, essential. Also, Example~\ref{ex:second} from Section~\ref{sec:cont_metric} indicates that, in theorems of Section~\ref{sec:cont_normed}, the convexity condition cannot be omitted. Finally, Section~\ref{sec:finite_dist} proves Theorem~\ref{thm_lem_2} that in real normed spaces of dimension not greater than $2$, the convexity of $A$ and $B$ implies $d_H\bigl(F_{r_1},F_{r_2}\bigr)<\infty$ for any nonempty $F_{r_1}$ and $F_{r_2}$. Therefore, in such spaces, the finiteness of distances between the images of $F_r$ can be omitted as a separate requirement, see Corollary~\ref{cor:without_fin_dist}.

The main results of this article are Theorems~\ref{thm:rightSemicont},~\ref{Lipschitz},~\ref{cont_criterion}, and~\ref{thm_lem_2}, as well as Example~\ref{ex:infinity_distance}.

\section{Results}\label{sec2}

For convenience, in the case of a metric space $(X,\rho)$, the distance between two points $a,b\in X$ will be denoted by $|a\,b|$ instead of $\rho(a, b)$. Also, instead of $(X,\rho)$, we will simply write $X$.

In the case of a linear space $X$ over the field $\mathbb{R}$, for any two points $a,b\in X$, the following standard notations will be used in many places in the text:
\begin{gather*}
(a,b):=\bigl\{(1-\lambda)a+\lambda b:\lambda\in(0,1)\bigr\}; \\
(a,b]:=\bigl\{(1-\lambda)a+\lambda b:\lambda\in(0,1]\bigr\}; \\
[a,b):=\bigl\{(1-\lambda)a+\lambda b:\lambda\in[0,1)\bigr\}; \\
[a,b]:=\bigl\{(1-\lambda)a+\lambda b:\lambda\in[0,1]\bigr\}.
\end{gather*}

\subsection{Necessary definitions and statements}\label{Definitions}

\subsubsection{Accessibility Lemma}
The following lemma can be found, for example, in~\cite[Lemma 5.1.3]{CalTech}.

\begin{lem}\label{lem:Access}
Let $X$ be a topological real linear space and $M\subset X$ be convex. Then for any $x\in\Int M$ and $y\in\Cl M$, we have $[x,y)\subset\Int M$.
\end{lem}

\begin{proof}
Since $x\in\Int M$, there exists an open neighborhood $U^x$ of $x$ such that $U^x\subset\Int M$. Since $M$ is convex, the cone $C$ with the vertex $y$ and the base $U^x$ belongs to $M$. Clearly that $C\setminus\{y\}$ is contained in $\Int M$, which proves the lemma.
\end{proof}

\subsubsection{Intersection of a segment with the boundary of a set}
\begin{lem}\label{lem:top}
Let $M$ be a nonempty subset of a connected topological space $X$, not coinciding with $X$. Then $\partial M\ne\emptyset$.
\end{lem}

\begin{proof}
If $\partial M=\emptyset$ then $M=\Int M$ is open and nonempty, and the same is true for $X\setminus M$, that contradicts to connectivity of $X$.
\end{proof}

We will also need the following

\begin{lem}[{\cite[Lemma~3]{Wills}}]\label{light}
Let $B$ be a nonempty subset of a real normed space $X$ such that $X\setminus\Int B\ne\emptyset$. Then for any $a\in X\setminus\Int B$ and $b\in \Cl B$ we have $[a,b]\cap\partial B\ne\emptyset$.
\end{lem}

\subsubsection{Definition and properties of $r$-neighborhoods}
Most statements in this subsection are well known, see for example~\cite{TuzLections}.

\begin{dfn}\label{dfn:one}
For any subset $A$ of a metric space $X$ and any point $p\in X$, the \textit{distance from $p$ to $A$} is the value
$$
|p\,A|=\inf\bigl\{|p\,a|:a\in A\bigr\}.
$$
In particular, for empty $A$ we have $|p\,A|=\infty$.
\end{dfn}

\begin{rk}
It is well known that for a nonempty subset $A$ of a metric space $X$, the mapping $x\mapsto|x\,A|$ is continuous.
\end{rk}

\begin{dfn}\label{dfn:balls}
Let $A$ be a subset of a metric space $X$ and $0\le r<\infty$. The subsets
$$
B_r(A)=\{p\in X:|p\,A|\le r\},\ \ U_r(A)=\{p\in X:|p\,A|<r\}
$$
are called, respectively, the \emph{closed\/} and the \emph{open ball with center at $A$ and radius $r$} or the \emph{closed\/} and the \emph{open $r$-neighborhood of $A$}.
\end{dfn}

\begin{rk}\label{rk:2}
According to Definition~$\ref{dfn:one}$, for any $0\le r<\infty$,
$$
B_r(\emptyset)=U_r(\emptyset)=\emptyset.
$$
\end{rk}

In the case $A = \{a\}$, the notations $B_r\bigl(\{a\}\bigr)$ and $U_r\bigl(\{a\}\bigr)$ will be replaced for brevity by $B_r(a)$ and $U_r(a)$, respectively.

\begin{lem}\label{lem:Closed}
Let $A$ be a subset of a metric space $X$. Then for any $0\le r<\infty$, the set $B_r(A)$ is closed, and $B_0(A)=\Cl A$.
\end{lem}

\begin{proof}
If $A=\emptyset$, then according to Remark~\ref{rk:2}, $B_r(A)=\emptyset$ is closed. For $A\ne\emptyset$, it follows immediately from continuity of the function $x\mapsto|x\,A|$.
\end{proof}

\begin{rk}
Notice that $\cup_{a\in A}B_r(a)\subset B_r(A)$, although equality is not always achieved. For instance, let $A=\{1/n\}_{n\in\mathbb{N}}\subset\{-1\}\cup\{1/n\}_{n\in\mathbb{N}}=:X\subset\mathbb{R}$, then $B_1(A)=X$, but $\cup_{a\in A}B_1(a)=X\setminus\{-1\}$.
\end{rk}

\begin{lem}\label{lem:Union}
Let $A$ be a subset of a metric space $X$. Then for any $0<r<\infty$, we have $U_r(A)=\bigcup\limits_{a\in A}U_r(a)$.
\end{lem}

\begin{proof}
If $A=\emptyset$, then according to Remark~\ref{rk:2}, $U_r(A)=\emptyset$. On the other hand, $\bigcup\limits_{a\in\emptyset}U_r(a)=\emptyset$.

For nonempty $A$, we have
\begin{multline*}
x\in U_r(A)\Leftrightarrow|x\,A|=\inf_{a\in A}|x\,a|<r\Leftrightarrow\exists a\in A,|x\,a|<r\Leftrightarrow\\
\exists a\in A, x\in U_r(a)\Leftrightarrow x\in\cup_{a\in A}U_r(a),
\end{multline*}
which completes the proof.
\end{proof}

From Lemma~\ref{lem:Union} we get

\begin{cor}\label{lem:Opened}
For any subset $A$ of a metric space $X$ and any $0<r<\infty$, the set $U_r(A)$ is open.
\end{cor}

\begin{lem}\label{lem:Opened_in_Closed}
Let $A$ be a nonempty convex subset of a real normed space $X$. Then for any $0<r<\infty$, $U_r(A)=\Int B_r(A)$.
\end{lem}

\begin{proof}
According to Definition~\ref{dfn:balls} and Corollary~\ref{lem:Opened}, $U_r(A)\subset\Int B_r(A)$. Assume $\Int B_r(A)\setminus U_r(A)\ne\emptyset$ and take an arbitrary $x\in\Int B_r(A)\setminus U_r(A)$, then $|x\,A|=r$. Since $x\in\Int B_r(A)$, there exists $s>0$ such that $B_s(x)\subset B_r(A)$. Take an arbitrary $a\in A$, then the segment $[a,x]$ is not degenerate. Extend the segment over the point $x$ up to some $y$ such that $|x\,y|=s$. Then $y\in B_s(x)\subset B_r(A)$ and $[a,y]\subset B_r(A)$ because $B_r(A)$ is convex by Lemma~\ref{lem:convexity}. In addition, $|y\,a|=|x\,a|+s$. Consider an arbitrary $b\in A$, then $[a,b]\subset A$ because $A$ is convex.

Let $f$ be the homothety with center at $a$ and coefficient $k=|x\,a|/|y\,a|<1$. Then $f(y)=x$ and we put $c=f(b)$. Since $[a,b]\subset A$ and $0<k<1$, we have $c\in A$. Notice that we can choose $a$ and $b$ independently and such that $|x\,a|=r+\alpha$, where $\alpha$ is non-negative and arbitrary close to $0$, and for any $\beta>0$ one can choose $b$ such that $|y\,b|\in[0,r+\beta)$. We get
$$
|x\,c|=k\,|y\,b|=\frac{r+\alpha}{r+\alpha+s}\,|y\,b|\le\frac{r+\alpha}{r+\alpha+s}\,(r+\beta).
$$
For $\alpha\to0$ and $\beta\to0$, the right hand side of the previous equation tends to $\frac r{r+s}\,r<r$, hence we can choose $\alpha$ and $\beta$ such that $|x\,c|<r$, a contradiction with $|x\,A|=r$.
\end{proof}

\begin{lem}\label{lem:convexity}
Let $A$ be a nonempty convex subset of a real normed space $X$. Then for $0\le r<\infty$, the sets $B_r(A)$ and $U_r(A)$ are convex.
\end{lem}

\begin{proof}
To start with, we prove that $U_r(A)$ is convex. Take arbitrary $x,y\in U_r(A)$ then, by Lemma~\ref{lem:Union}, there exists $a,b\in A$ such that $|a\,x|<r$ and $|b\,y|<r$. Then for any $\lambda\in[0,1]$, we have $(1-\lambda)a+\lambda b\in A$ due to convexity of $A$, and
$$
\|(1-\lambda)a+\lambda b-(1-\lambda)x+\lambda y\|\le(1-\lambda)|a\,x|+\lambda|b\,y|<r,
$$
hence $(1-\lambda)x+\lambda y\in U_r(A)$.

To prove the result for $B_r(A)$, it suffices to change $a$ and $b$ with sequences $a_i\in A$ and $b_i\in A$ such that $|a_i\,x|\to|x\,A|\le r$ and $|b_i\,y|\to|y\,A|\le r$, proceed in the same way as above and pass to the limit.
\end{proof}

Applying Lemma~\ref{lem:convexity} for $r=0$, we get the following

\begin{cor}\label{cor:closed_conv}
In a real normed space $X$, the closure of a nonempty convex set is convex.
\end{cor}

The next lemma is standard, see~\cite[p.~22]{TuzLections}.

\begin{lem}\label{sum_0}
Let $A$ be a nonempty subset of a metric space $X$, and $r, r' \in [0, \infty)$. Then $B_r\bigl(B_{r'}(A)\bigr)\subset B_{r+r'}(A)$.
\end{lem}

\begin{rk}
In general metric spaces, we cannot change the inclusion to equality in Lemma~\ref{sum_0}. For example, let $A=\{-1\}\subset\{-1\}\cup\{1/n\}_{n\in\mathbb{N}}=:X\subset\mathbb{R}$, then $B_1(A)=A$, hence $B_1(B_1)(A)=A$. However, $B_2(A)=X$.
\end{rk}

\begin{lem}\label{sum_1}
Let $A$ be a nonempty subset of a real normed space $X$, and $r, r' \in [0, \infty)$. Then $B_r\bigl(B_{r'}(A)\bigr) = B_{r+r'}(A)$.
\end{lem}

\begin{proof}
Lemma~\ref{sum_0} implies $B_r\bigl(B_{r'}(A)\bigr) \subset B_{r+r'}(A)$. Let us prove the converse inclusion. Consider an arbitrary $p\in B_{r+r'}(A)$, then for any $\varepsilon>0$ there exists $a\in A$ such that $|p\,a|<r+r'+\varepsilon$. If $|p\,a|\le r'$, then we put $b=p$, hence we got $b\in B_{r'}(A)$ and $|b\,p|=0<r+\varepsilon$. If $|p\,a|>r'$, we choose $b\in[p,a]$ such that $|a\,b|=r'$, hence $b\in B_{r'}(A)$ and $|b\,p|<r+\varepsilon$. Since we can choose $\varepsilon>0$ in arbitrary way, we got $\bigl|p\,B_{r'}(A)\bigr|\le r$, thus $p\in B_r\bigl(B_{r'}(A)\bigr)$.
\end{proof}

\begin{proposition}\label{as:ray}
Let $A$ be a nonempty convex subset of a real normed space $X$, $x\not\in\Cl A$, and $p\in A$. Then the ray
$$
\ell := (1-\lambda) p + \lambda x,\ \ \lambda\in [0, \infty),
$$
does not lie in any ball $B_r(A)$ for all $0\le r<\infty$.
\end{proposition}

\begin{proof}
Since $x\not\in\Cl A$, we have $r:=|x\,A|>0$ and hence $|p\,x|\ge r>0$. Choose an arbitrary $k>1$ and $y_k\in\ell$ such that $|p\,y_k| = k|p\,x|$. Due to Corollary~\ref{cor:closed_conv}, the set $\Cl A$ is convex. Then $y_k\ne\Cl A$, otherwise $x\in\Cl A$. Then $|y_k\,A|=:\rho>0$.

For any $\varepsilon>0$, let $b\in A$ be such that $|y_k\,b|<\rho+\varepsilon$. Consider homothety $f$ with center $p$ and coefficient $1/k$, then $f(y_k)=x$. Let $a=f(b)$, then $a\in[b,p]\subset A$ because $A$ is convex. Hence
$$
r=|x\,A|\le|x\,a|=|y_k\,b|/k<(\rho+\varepsilon)/k,
$$
thus $\rho>k\,r-\varepsilon$. Since $\varepsilon$ is chosen arbitrarily, we have $|y_k\,A|=\rho\ge k\,r$, thus $|y_k\,A|\to\infty$ as $k\to\infty$, that completes the proof.
\end{proof}

\subsubsection{Rays in convex set}
\begin{lem}\label{lem_invar}
Let $X$ be a real normed space, and let for a sequence $\{x_i\}\subset X$ and a point $p\in X$, it holds $\|x_i - p\|\to \infty$ and $\frac{x_i - p}{\|x_i - p\|}\to y\in X$ as $i\to \infty$. Then for any point $q\in X$, we have $\|x_i - q\|\to \infty$ and $\frac{x_i - q}{\|x_i - q\|}\to y$ as $i\to \infty$.
\end{lem}

\begin{proof}
Choose a new coordinate system by taking $p$ as the new origin. Since the shift $x\mapsto x-p$ is an isometry, it suffices to prove the lemma in the new coordinates, i.e., without loss of generality we can assume that $p=0$. Moreover, starting from sufficiently large $i$, we can also assume that both $x_i$ and $x_i-q$ do not vanish for all $i$.

Notice that $\|x_i-q\|/\|x_i\|\to1$ and thus $\|x_i\|/\|x_i-q\|\to1$ as $i\to\infty$, because
$$
1-\frac{\|q\|}{\|x_i\|}\le\frac{\|x_i-q\|}{\|x_i\|}\le1+\frac{\|q\|}{\|x_i\|}
$$
for each $i$. Hence
$$
\frac{x_i-q}{\|x_i-q\|}=\frac{\|x_i\|}{\|x_i-q\|}\biggl(\frac{x_i}{\|x_i\|}-\frac{q}{\|x_i\|}\biggr)\to y,
$$
that completes the proof.
\end{proof}

\begin{proposition}\label{two_abz}
Let $A$ be a nonempty convex subset of a real normed space $X$, and suppose that some ray $\ell$ emitted from a point $x\in A$ belongs to $A$. Then for any $y\in A$ and any $z\in[x, y)$, the ray $\ell'$ emitted from $z$ and codirected with $\ell$ belongs to $A$ as well. Moreover, if $A$ is closed, then we can take $z$ from $[x,y]$.
\end{proposition}

\begin{proof}
To start with, we consider the case $z\in[x,y)$. It is easy to see that the cone $y\ell$ with the vertex $y$ and the base $\ell$ contains each ray $\ell'$. Since $y\ell$ consists of segments connecting points from $A$, it is contained in $A$, together with $\ell'$.

Now, assume that $A$ is closed and take $z=y$. Consider an arbitrary $p\in\ell'$, and put $r=|z\,p|$. Choose a sequence $x_i\in\ell$ such that $|x\,x_i|\to\infty$ as $i\to\infty$ and $|z\,x_i|>r$ for all $i$. By Lemma~\ref{lem_invar}, the directions of vectors $x_i - z$ tends to the direction of the rays $\ell$ and $\ell'$, thus the points $p_i\in[z,x_i]\subset A$ such that $|z\,p_i|=r$ tends to $p$. Since $A$ is closed, then $p\in A$.
\end{proof}

\subsubsection{Hausdorff distance}
For any set $X$, denote by $\mathcal{P}_0(X)$ the collection of all nonempty subsets of $X$.

\begin{dfn}\label{distance_H}
Given a metric spaces $X$, we endow $\mathcal{P}_0(X)$ with the \emph{Hausdorff distance $d_H$} defined as follows: for any $A,B\in\mathcal{P}_0(X)$,  we put
$$
d_H(A, B) = \inf\bigl\{r:A\subset B_r(B),\,B\subset B_r(A)\bigr\}.
$$
\end{dfn}

Notice that $d_H$ is non-negative, symmetric, and satisfies the triangle inequality, but can be zero for $A\ne B$ and can be infinite.

\begin{proposition}[{\cite[Corollary~5.30]{TuzLections}}]\label{Select}
For an arbitrary metric space $X$ and an arbitrary decreasing sequence of non-empty subsets $A_1\supset A_2\supset\ldots$ of $X$, where $A_1$ is compact, we have $A_k\xrightarrow{d_H}\bigcap\limits_n\Cl A_n\ne\emptyset$.
\end{proposition}

\subsection{Main part}\label{Main}

Let $X$ be a metric spaces and $\mathcal{P}_0(X)$ be endowed with the Hausdorff distance $d_H$. Then, for any $A, B\in\mathcal{P}_0(X)$ and any $r\in\bigl(|A\,B|,\infty\bigr)\subset\mathbb{R}$, we put $F_r(A, B)=B_r(A)\cap B$. Note that it holds $F_r(A, B)\in \mathcal{P}_0(X)$. Thus, henceforth we will measure the distance between sets $F_{r_1}(A, B)$ and $F_{r_2}(A, B)$ using the Hausdorff distance $d_H$, and we will understand continuity of $F_r(A, B)$ in $r$ with respect to $d_H$.

\subsubsection{Continuity of $F_r$ in metric spaces}\label{sec:cont_metric}
We start from the first part of our main result. If $X$ is a metric space, we denote by $\mathcal{P}_\comp(X)$ the subfamily of $\mathcal{P}_0(X)$ consisting of all compact sets.

\begin{thm}\label{thm:rightSemicont}
Let $A$ and $B$ be non-empty subsets of a metric space $X$, and suppose that $B\in \mathcal{P}_\comp(X)$. Then the mapping $F_r(A,B)$ is right semicontinuous in $r$.
\end{thm}

\begin{proof}
For brevity, we put $F_r=F_r(A,B)$. Suppose otherwise that there exists $s\in \bigl(|A\,B|,\infty\bigr)$ such that $F_r$ is not right semicontinuous at $r=s$. Consider an arbitrary decreasing sequence $s_n$ such that $s_n\to s$ as $n\to\infty$. By Lemma~\ref{lem:Closed} and compactness of $B$, the sets $F_{s_n}$ are compact for all $n$, hence, by Proposition~\ref{Select}, we have $F_{s_n}\xrightarrow{d_H}\cap_nF_{s_n}=:F$. By assumption, $F\ne F_s$. Notice that
$$
F_s\subset\bigcap\limits_nF_{s_n}=F=\Bigl(\bigcap\limits_nB_{s_n}(A)\Bigr)\cap B,
$$
and since $F_s\subset F$, there exists $x\in F\setminus F_s$. For such $x$, we have $x\in B$, $x\in\cap_nB_{s_n}(A)$, and $x\not\in F_s=B_s(A)\cap B$, hence $x\not\in B_s(A)$ and $|x\,A|>s$. However, there exists $m$ such that $s_m<|x\,A|$, and therefore $x\not\in B_{s_m}(A)$, a contradiction with $x\in\cap_nB_{s_n}(A)$.
\end{proof}

The compactness condition of $B$ is essential in Theorem~\ref{thm:rightSemicont}.

\begin{exa}\label{ex:first}
Let $X=\mathbb{R}\setminus\{1\}$, $A=\{0\}$, and $B=\{0\}\cup\{1+1/n\}_{n=1}^\infty$. The set $B$ is closed and bounded, but not compact. It is easy to see that in this case, the mapping $F_r$ is right discontinuous at $r=1$.
\end{exa}

We can also construct a similar example in a normed space.

\begin{exa}\label{ex:first_normed}
Let $X$ be a real normed infinite-dimensional space with the origin $0$ and a countable basis $e_1,\,e_2,\ldots$. Let $A=\{0\}$ and $B=\{0\}\cup\left\{\left(1+1/n\right)e_n\right\}_{n=1}^\infty$. The set $B$ is closed and bounded but not compact. It is easy to see that in this case, the mapping $F_r$ is right discontinuous at $r=1$.
\end{exa}

The next example shows that the compactness condition of $B$ does not imply left semicontinuity.

\begin{exa}\label{ex:second}
In the Euclidean plane, let $A$ be a closed rectangle whose sides are parallel to the coordinate axes in the standard basis $e_1=(1,0)$, $e_2=(0,1)$, and the set $B$ be some non-convex compact subset placed over $A$, see Figure~\ref{discont}.
\begin{figure}[h]
\centering
\includegraphics[width=0.5\textwidth]{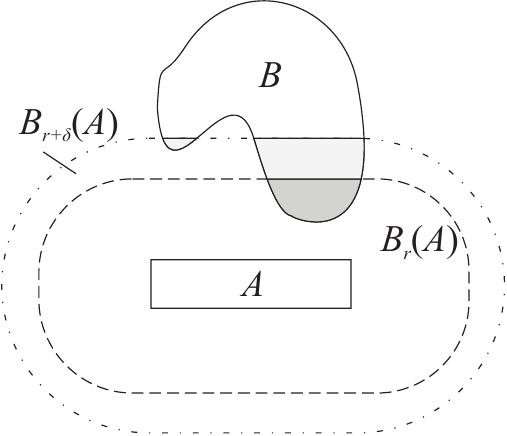}
\caption{The mapping $F_r(A,B)$ is discontinuous from the left provided by $B$ is compact}\label{discont}
\end{figure}
Note that as the radius $r$ increases, at some moment $r=r_0$ the set $F_r(A, B)$ acquires a new connected component. This implies that at the point $r=r_0$ the mapping $F_r$ has a left discontinuity in the Hausdorff distance.

We will prove the left semicontinuity of the mapping $F_r$ provided $A$ and $B$ are not necessarily compact, but are convex subsets of a normed space.
\end{exa}

\subsubsection{Continuity of $F_r$ in normed spaces}\label{sec:cont_normed}

If $X$ is a real normed space, we denote by $\mathcal{P}_\conv(X)$ the subfamily of $\mathcal{P}_0(X)$ consisting of all convex sets.

\begin{proposition}\label{conv_func}
Let $X$ be real normed space, and $A,\,B\in\mathcal{P}_\conv(X)$. Let the mapping $f\colon X\times\bigl(|A\,B|,\infty\bigr)\to[0,\infty)$ be defined by the rule $f(x,r)=\bigl|x\,F_r(A, B)\bigr|$. Then $f$ is convex in $r$.
\end{proposition}

\begin{proof}
For brevity, put $F_{\rho}:=F_{\rho}(A,B)$. Fix $x\in X$ and set $r=(1-t)r_1+t\,r_2$, where $r_1,\,r_2\in\bigl(|A\, B|,\infty\bigr)$. Upto the end of this paragraph, all statements that include $i$ are valid for both $i=1,\,2$. Take any $\varepsilon>0$. There exists $\widetilde{y}_i\in F_{r_i}$ such that $|x\,\widetilde{y}_i|<f(x,r_i)+\varepsilon/2$. Since $r_i\in\bigl(|A\,B|,\infty\bigr)$, we have $U_{r_i}(A)\cap B\ne\emptyset$. Let $z_i\in U_{r_i}(A)\cap B$. Since $A$ and $B$ are convex, Lemmas~\ref{lem:convexity},~\ref{lem:Opened_in_Closed}, and~\ref{lem:Access} imply $[z_i,\widetilde{y}_i)\subset U_{r_i}(A)\cap B$. So, there exists $y_i\in U_{r_i}(A)\cap B$ such that $|\widetilde{y}_i\,y_i|<\varepsilon/2$. By triangle inequality, we get $|x\,y_i|<f(x,r_i)+\varepsilon$. In addition, since $y_i\in U_{r_i}(A)$, there exists $a_i\in A$ such that $|y_i\,a_i|<r_i$.

Let $y=(1-t)y_1+t\,y_2$ and $a=(1-t)a_1+t\,a_2$. By the convexity of $A$ and $B$, we have $y\in B$ and $a\in A$. Further, $|y\,a|\le(1-t)|y_1\,a_1|+t|y_2\,a_2|<(1-t)r_1+t\,r_2=r$. From this we obtain $y\in U_r(A)$, and therefore, $y\in F_r$.

Thus, we have
$$
|x\,y|\le(1-t)|x\,y_1|+t|x\,y_2|<(1-t)f(x,r_1)+t\,f(x,r_2)+\varepsilon,
$$
therefore,
$$
f\bigl(x,(1-t)r_1+t\,r_2\bigr)=|x\,F_r|\le|x\,y|<(1-t)f(x,r_1)+t\,f(x, r_2)+\varepsilon.
$$
From the arbitrariness of $\varepsilon>0$ it follows that the mapping $f(x,r)$ is convex in $r$.
\end{proof}

\begin{thm}\label{Lipschitz}
Let $X$ be a real normed space, and $A, B\in \mathcal{P}_\conv(X).$ Suppose that for some $(Q,T)\subset\bigl(|A\,B|, \infty\bigr)$ and any $s_1, s_2\in(Q,T)$, it holds $d_H\bigl(F_{s_1}(A, B), F_{s_2}(A, B)\bigr)<\infty.$ Then, for any $q, t\in(Q,T)$, there exists $M>0$ such that the mapping $F_r(A, B)$ is $M$-Lipschitz in $r$ on $[q, t]$.
\end{thm}

\begin{proof}
For brevity, put $F_{\rho}:=F_{\rho}(A, B)$. Choose arbitrary $r,r'\in(Q,T)$, $r<r'$, and $q,q',t\in(Q,T)$ such that $Q<q'<q\le r<r'\le t<T$, then $d:=d_H(F_{q'},F_t)<\infty$. Since for each $\rho\in[q',t]$ we have $F_{q'}\subset F_\rho\subset F_t$, it holds $d_H(F_{q'},F_\rho) = \sup_{x\in F_\rho} |x\, F_{q'}|\le d$.

Since $r\in(q',r')$, there exists and unique $\delta\in (0, 1)$ such that $r=(1-\delta)r'+\delta q'$, hence $\delta=(r'-r)/(r'-q')<(r'-r)/(q-q')$. Put $M=d/(q - q')$.

By Proposition~\ref{conv_func}, for each $x\in F_{r'}$ it holds
$$
|x\,F_r|\le(1-\delta)|x\,F_{r'}|+\delta|x\,F_{q'}|=\delta|x\,F_{q'}|\le\delta d<\frac d{q-q'}(r'-r)=M(r'-r).
$$
Hence, for any $r,r'\in[q, t]$, $r<r'$, we obtain $d_H(F_r,F_{r'})=\sup_{x\in F_{r'}}|x\,F_r|\le M(r'-r)$, i.e., $F_r$ is $M$-Lipschitz in $r$ on $[q, t]$.
\end{proof}

Theorem~\ref{Lipschitz} directly implies the following

\begin{cor}\label{cont}
Let $X$ be a real normed space, and $A, B\in \mathcal{P}_\conv(X)$. Suppose that for some $(Q,T)\subset\bigl(|A\,B|,\infty\bigr)$ and any $s_1,s_2\in(Q,T)$, it holds $d_H\bigl(F_{s_1}(A, B),F_{s_2}(A, B)\bigr)<\infty$. Then, the mapping $F_r(A, B)$ is continuous in $r$ on $(Q,T)$.
\end{cor}

This leads to the following continuity criterion.

\begin{thm}[Continuity Criterion of $F_r$ in normed spaces]\label{cont_criterion}
Let $X$ be a real normed space, and $A,B\in\mathcal{P}_\conv(X)$. Then the mapping $F_r(A,B)$ is continuous in $r$ on some interval $(Q,T)\subset\bigl(|A\,B|,\infty\bigr)$ if and only if for any $s_1,s_2\in(Q,T)$, it holds
$$
d_H\bigl(F_{s_1}(A,B),F_{s_2}(A,B)\bigr)<\infty.
$$
\end{thm}

\begin{proof}
If for any $s_1,s_2\in(Q,T)$ it holds $d_H\bigl(F_{s_1}(A,B),F_{s_2}(A,B)\bigr)<\infty$, then the mapping $F_r$ is continuous in $r$ on $(Q,T)$ by Corollary~\ref{cont}. Thus, it remains to prove the converse statement.

For brevity, put $F_{\rho}=F_{\rho}(A,B)$. Let $F_r$ be continuous on $(Q,T)$. This means that for each $s\in(Q,T)$ and any $\varepsilon>0$, there exists $\delta_s$ such that $I_s:=(s-\delta_s,s+\delta_s)\subset(Q,T)$ and for any $r\in I_s$ it holds $d_H(F_s,F_r)<\varepsilon$. This implies that for each $s\in(Q,T)$, the function $f_s(r)=d_H(F_s,F_r)$ is finite on the interval $I_s$.

Now we take an arbitrary $s_1,s_2\in(Q,T)$, say $s_1<s_2$, and consider the open covering $\{I_s\}_{s\in[s_1,s_2]}$. By compactness of $[s_1,s_2]$ we can take a finite subcovering $I_1,\ldots,I_n$ such that $I_{k-1}\cap I_k\ne\emptyset$ for each $k=2,\ldots,n$, and there exists a sequence $s_1=t_1<t_2<\cdots<t_n=s_2$ with $t_k\in I_k$. Since consecutive $I_k$ intersect each other, we get $d_H(F_{t_{k-1}},F_{t_k})<\infty$, and, by the triangle inequality, we obtain $d_H(F_{s_1},F_{s_2})<\infty$, that completes the proof.
\end{proof}

\subsubsection{Finiteness of Hausdorff distance between the images of $F_r$}\label{sec:finite_dist}
Note that, in general, the convexity of the sets $A$ and $B$ does not imply
$$
d_H\bigl(F_r(A, B), F_s(A,B)\bigr)<\infty
$$
for any $r,s\in\bigl(|A\,B|,\infty\bigr)$, $r\ne s$. The following example illustrates that.

\begin{exa}\label{ex:infinity_distance}
Fix in $\mathbb{R}^3$ the standard basis $e_1=(1,0,0)$, $e_2=(0,1,0)$, $e_3=(0,0,1)$. Let
$$
A=\bigl\{(x,y,z):z\le 0\bigr\}\ \ \text{and}\ \
B=\bigl\{(x,y,z):y\ge x^2/z,\,z>0\bigr\}\cup\bigl\{(0,y,0):y\ge 0\bigr\}.
$$
Analogically, for brevity, put $F_{\rho}:=F_{\rho}(A,B)$. Figure~\ref{parabolas} illustrates the sets $F_r$ for two different values of $r$.
\begin{figure}[h]
\centering
\includegraphics[width=0.6\textwidth]{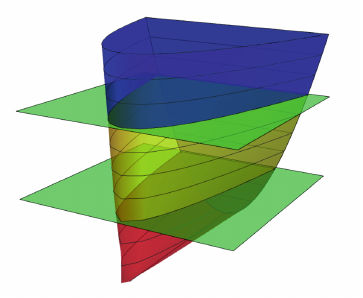}
\caption{$F_{r_1}$ is the red set, $F_{r_2}$ is the union of red and yellow sets}\label{parabolas}
\end{figure}

We show that $B$ is convex. Pass to the new basis by rotating the vectors $e_2$ and $e_3$ around the abscissa by $-\pi/4$. This transforms the quadratic form $x^2-yz$ into $x^2+y^2/2-z^2/2$, thus, in the new basis the inequality $x^2-yz\le0$ takes the form
$x^2+y^2/2-z^2/2\le0$. The solution of the latter one consists of two convex filled elliptic cones, one in the half-space $z\ge0$, and another one in $z\le0$.

Further, the inequality $z\ge0$ in the previous basis is transformed into $z\ge y$ in the new basis. It is easy to see that the intersection of the double cone $x^2+y^2/2-z^2/2\le0$ with the half-space $z\ge y$ yields the cone in $z\ge0$, that completes the proof of the convexity of $B$. The space $A$ is convex as each half-space.

Now we need to show that for any $r_1,\,r_2\in[0,\infty)$, $r_1<r_2$, the Hausdorff distance between $F_{r_1}$ and $F_{r_2}$ is infinite. Let $P_i:=\partial B_{r_i}(A)\cap B$ be the planar section of the set $B$ by the plane $\pi_i:=\partial B_{r_i}(A)$. Notice that $P_i$ is the epigraph of a parabola, and $d_H(F_{r_1},F_{r_2})=\sup_{u\in P_2}|u\,P_1|$. Denote by $P'_2$ the orthogonal projection of $P_2$ to the plane $\pi_1$. Clearly $\sup_{u\in P_2}|u\,P_1|\ge\sup_{u\in P'_2}|u\,P_1|$. It remains to note that $P'_2$ and $P_1$ are epigraphs of parabolas $p_1:=\bigl\{y=(1/r_1)x^2\bigr\}$ and $p'_2:=\bigl\{y=(1/r_2)x^2\bigr\}$ in the plane $z=r_1$, but
$$
\sup_{u\in P'_2}|u\,P_1|=\sup_{u\in p'_2}|u\,p_1|=\infty,
$$
because $r_1\ne r_2$. Thus, we get $d_H(F_{r_1},F_{r_2})=\infty$ for all admissible $r_1\ne r_2$.
\end{exa}

However, in the case of convex sets in real normed spaces of dimension not exceeding $2$, it is impossible to construct an example of the mapping $F$ for which there would exist $r,\,s\in\bigl(|A\,B|,\infty\bigr)$ with the condition $d_H(F_r,F_s)=\infty$. Namely, the following theorem holds.

\begin{thm}\label{thm_lem_2}
Let $X$ be a real normed space, dimension of $X$ does not exceed $2$, and $A,B\in\mathcal{P}_\conv(X)$. Then for any $r,\,s\in\bigl(|A\,B|,\infty\bigr)$, it holds
$$
d_H\bigl(F_r(A,B),F_s(A,B)\bigr)<\infty.
$$
\end{thm}

\begin{proof}
Since zero-dimensional and one-dimensional spaces are subspaces of two-dimensional ones, it suffices to prove the theorem in the case when the dimension of the space is $2$.

So, let the dimension of $X$ be $2$. For brevity, put $F_{\rho}=F_{\rho}(A,B)$. Assume that there exist $r, s\in \bigl(|A\,B|, \infty\bigr)$, $r>s$, for which
\begin{equation}\label{fardist}
d_H(F_r, F_s)=\infty.
\end{equation}
Since $A,B\in\mathcal{P}_\conv(X)$, Lemma~\ref{lem:convexity} implies $F_r,F_s\in\mathcal{P}_\conv(X)$.

By~(\ref{fardist}), there exists a sequence
\begin{equation}\label{closed}
\{x_i\}\subset F_r\setminus F_s
\end{equation}
such that
\begin{equation}\label{inf}
|x_i\, F_s|\to\infty\ \ \text{as $i\to\infty$}.
\end{equation}
From~(\ref{closed}) we get $\{x_i\}\subset B$. At the same time, for all $i$ we have $x_i\not\in F_s=B_s(A)\cap B$ and, hence, $x_i\not\in B_s(A)$. Therefore,
\begin{equation}\label{balls}
\{x_i\}\subset B_r(A)\setminus B_s(A).
\end{equation}
In addition, by Lemma~\ref{sum_1}, it holds
\begin{equation}\label{nested_balls}
B_{r-s}\bigl(B_{s}(A)\bigr)=B_r(A).
\end{equation}

Let
\begin{equation}\label{point_a}
a\in U_s(A)\cap B\subset F_s.
\end{equation}
Since all properties of the sets $A$ and $B$ used in this proof are invariant under translations, without loss of generality we will further assume that the point $a$ coincides with the origin $0$ of the space $X$. By (\ref{balls}) and (\ref{point_a}), the sequence $\bigl\{x_i/\|x_i\|\bigr\}$ is correctly defined.

Due to the compactness of the unit sphere in each finite-dimensional space, there exists a subsequence $\bigl\{x_{i_j}/\|x_{i_j}\|\bigr\}$ converging to some point $y\ne0$. After removing elements and performing appropriate renumbering, we will further assume that the entire sequence $\bigl\{x_i/\|x_i\|\bigr\}$ converges to $y$.

\begin{lem}\label{lem:out_of_ray}
The ray $l_a:=\lambda y$, $\lambda\ge0$, contains at most finitely many elements from $\{x_i\}$.
\end{lem}

\begin{proof}
Assume that infinitely many elements from $\{x_i\}$ lie on the ray $l_a$. Then from~(\ref{inf}),~(\ref{balls}), and the convexity of $B_r(A)$, it follows that $l_a\subset B_r(A)$ and $l_a\not\subset B_s(A)$. But according to~(\ref{point_a}), the ray $l_a$ starts from $B_s(A)$, which cannot be by Proposition~\ref{as:ray} in view of~(\ref{nested_balls}) and the convexity of $B_s(A)$, a contradiction.
\end{proof}

To describe the geometric locus of points by means of angles, we fix some inner product in $X$, and choose an orthonormal basis $(e_1,e_2)$ such that the direction of the ray $l_a$ coincides with $e_1$. Recall that the \emph{oriented angle\/} between a vector $x\in X\setminus(-l_a)$ and the vector $e_1$ is the ordinary angle multiplied by $-1$ provided the coordinate of the vector $x$ with respect to $e_2$ is negative.

In what follows, we will frequently use the fact that any two norms on a finite-dimensional vector space are equivalent. In particular, if for a sequence of vectors one of the norms tend to infinity or to $0$, then the same holds for the second norm. Also, since such norms generate the same topology, then continuous dependance holds or not w.r.t. the both norms simultaneously.

Since $x_i/\|x_i\|\to y$ as $i\to\infty$, then, taking into account Lemma~\ref{lem:out_of_ray}, we come to the following agreement.
\begin{property}\label{corner}
Passing if necessarily to a subsequence, we further assume that all oriented angles $\alpha_i$ between the vectors $x_i$ and the ray $l_a$ (i.e., the vector $e_1$)
\begin{itemize}
\item are positive and belong to the interval $(0,\pi/2)$ (we change the direction of $e_2$ if necessarily),
\item strictly monotonically decreasing in $i$ and tend to $0$.
\end{itemize}
Also, if $x_i=x_i^1e_1+x_i^2e_2$, then, taking into account~(\ref{inf}), we assume that $x_i^1$, the norm $\|x_i\|$, and $\sqrt{(x_i^1)^2+(x_i^2)^2}$ are strictly monotonically increasing and tend to $\infty$. Moreover, we put $x_0=a$, fix some $\varepsilon>0$, and demand that for each $i=1,2,\ldots$, it holds $x_i^1-x_{i-1}^1>r-s+\varepsilon=:\delta$.
\end{property}

\begin{lem}\label{lem:l_in_M_r}
Under restrictions introduced above, we have $l_a\subset B_s(A)$. Moreover, for $h_i:=x_i^1e_1$ it holds $[x_i,h_i)\subset B$.
\end{lem}

\begin{proof}
Let $v\in l_a$ and $v\ne0$. Draw the line $l'$ perpendicular to $l_a$ through the point $v$. Since the dimension of $X$ equals $2$, Property~\ref{corner} implies that the line $l'$ intersects infinitely many segments $[a,x_i]$. For those $x_i$ with $[a,x_i]\cap l'\ne\emptyset$, we put
\begin{equation}\label{sj}
v_i= [a,x_i]\cap l',
\end{equation}
see Figure~\ref{lines}.
\begin{figure}[h]
\centering
\includegraphics[width=0.4\textwidth]{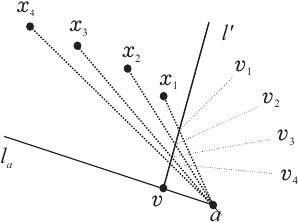}
\caption{Sequence $\{v_i\}$}\label{lines}
\end{figure}
By Property~\ref{corner}, $v_i\to v$ as $i\to \infty$. Due to the convexity of $F_r$, from~(\ref{closed}) and~(\ref{point_a}) it follows that $[a,x_i]\subset F_r$ for all $i$. Hence, by~(\ref{sj}), $\{v_i\}\subset B_r(A)$ and $\{v_i\}\subset B$. Since $B_r(A)$ is closed, $v\in B_r(A)$. The arbitrariness of $v\in l_a$, $v\ne a$, and $a\in B_r(A)$, imply
\begin{equation}\label{including}
l_a\subset B_r(A).
\end{equation}

Remind that the ray $l_a$ starts from $a\in B_s(A)$. If $l_a\not\subset B_s(A)$, then from~(\ref{nested_balls}) and the convexity of $B_s(A)$, by Proposition~\ref{as:ray} we have $l_a\not\subset B_r(A)$, which contradicts~(\ref{including}). Hence, $l_a\subset B_s(A)$.

To prove the second statement, we take $h_i$ as $v$, then all segments $[x_i,v_i]$ belong to $B$ by convexity reasons, thus $[x_i,h_i)\subset B$ because $v_i\to h_i$.
\end{proof}

Recall that $a\in U_s(A)$ by~(\ref{point_a}), and $x_1$ does not belong to the line $l$ passing through the ray $l_a$ by Property~\ref{corner}, thus $(a,x_1)\cap U_s(A)\ne\emptyset$, hence there exists $q\in(a,x_1)\cap U_s(A)$. Denote by $l_q$ the ray codirected with $l_a$ and starting at $q$. Since $l_a\subset B_s(A)$ by Lemma~\ref{lem:l_in_M_r}, and $B_s(A)$ is closed, we get $l_q\subset B_s(A)$ by Proposition~\ref{two_abz}.

Denote by $T_i$ the trapezoid $x_{i-1}x_ih_ih_{i-1}$ (if $i=1$ this trapezoid degenerates into the triangle $ax_1h_1$). From Lemma~\ref{lem:l_in_M_r} and convexity of $B$, we get $T_i\setminus l_a\subset B$. Denote by $T$ the union of all this trapezoids, thus $\Omega:=T\setminus l_a\subset B$.

Since $x_i\in B_r(A)$, and $B_r(A)=B_{r-s}\bigl(B_s(A)\bigr)$, for each $i$ there exists $p_i\in B_s(A)$ such that $|x_i\,p_i|<\delta$. Since $l_q\subset B_s(A)$, we can always choose $p_i$ from the same open half-plane bounded by $l$ where all $x_i$ are placed. Indeed, if $|x_i\,l_q|<\delta$, we can take $p_i$ from $l_q$.

Let $p_i=p_i^1e_1+p_i^2e_2$ and $q_i=p_i^1e_1$. Since for all $i\ge1$ we have $x_i^1-x_{i-1}^1>\delta$ by Property~\ref{corner}, we get $p_i^1>0$, i.e., $q_i\in l_a\setminus\{a\}$. Also, we have chosen $p_i$ such that $p_i^2>0$ for all $i\ge1$. Since both $p_i$ and $l_a$ belong to $B_s(A)$, the convexity of $B_s(A)$ implies $[p_i,q_i]\subset B_s(A)$ for $i\ge1$. Since $p_i^1>0$ and $p_i^2>0$ for $i\ge1$, the segment $[p_i,q_i]$ intersects $\Omega$. 

If $p_i\in \Omega$ then $p_i\subset B_s(A)\cap B=F_s$ and so we get $|x_i\, F_s| < \delta.$ If $p_i\notin \Omega$ then the segment $[p_i,q_i]\subset B_s(A)$ intersects the infinite polygonal line $x_0x_1x_2x_3\cdots\subset \partial\Omega\subset B$ at some point $t_i\in[x_{i-1},x_i]\cup[x_i,x_{i+1}]$, see Figure~\ref{ball}.
\begin{figure}[h]
\centering
\includegraphics[width=0.5\textwidth]{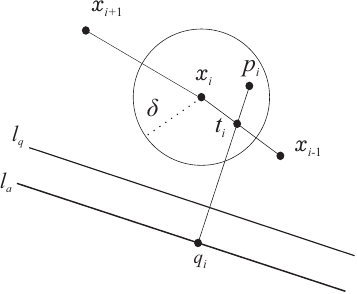}
\caption{Point $t_i$}\label{ball}
\end{figure}
Hence, $t_i\in F_s.$ Also due to Property~\ref{corner}, the angle between the segment $[x_i, x_{i+1}]$ and the ray $l_a$ tends to zero and so we obtain that $|p_i\,t_i|$ is bounded. This fact implies that $|x_i\, F_s|$ is bounded, since $|x_i\,p_i|<\delta$.

Thus, the value $|x_i\, F_s|$ is bounded for arbitrary $i$, that contradicts~(\ref{inf}). The proof of Theorem~\ref{thm_lem_2} is completed.
\end{proof}

From Corollary~\ref{cont} and Theorem~\ref{thm_lem_2} the following corollary directly follows.

\begin{cor}\label{cor:without_fin_dist}
Let $X$ be a real normed space, dimension of $X$ does not exceed $2$ and $A,B\in \mathcal{P}_\conv(X)$. Then the mapping $F_r(A, B)$ is continuous in $r$.
\end{cor}

\section{Conclusion}
This work studies the continuous dependence on $r \in \bigl[|A\,B|, \infty\bigr)$ of the mapping $F_r(A,B) = B_r(A) \cap B$ in metric and normed spaces, where $A$ and $B$ are nonempty subsets of these spaces. Such a mapping arises in geometric optimization problems with the Hausdorff metric, for example, when searching for subsets that minimize the sum of distances to given elements of the hyperspace $\mathcal{H}(X)$ over the space $X$, see~\cite{Gals_1}.

Furthermore, it would also be interesting to investigate the operation $F_r(A,B)$ with respect to continuous dependence on the subsets $A$ and $B$. Specifically, fix a radius $r \in \bigl(|A\,B|, \infty\bigr).$ The question is: will the mapping $F_r(A,B)$ be continuous at the point $(A, B) \in \mathcal{H}(X) \times \mathcal{H}(X)$, and if so, under what conditions and in what spaces? This question is the subject of possible future research.

\backmatter

\bmhead{Acknowledgements}
The work of Galstyan A.Kh. was carried out at Harbin Institute of Technology with the support of the Postdoctoral Research Start-up Funds (Funding card number: AUGA5710026125).

The work of Tuzhilin A.A. was supported by grant No. 25-21-00152 of the Russian Science Foundation, by National Key R\&D Program of China (Grant No. 2020YFE0204200), as well as by the Sino-Russian Mathematical Center at Peking University. Partial of the work by Tuzhilin A.A. were done in Sino-Russian Math. center, and he thanks the Math. Center for the invitation and the hospitality.



\end{document}